\documentclass[a4paper,12pt]{article}

\usepackage{latexsym}
\usepackage{amsmath}
\usepackage{amssymb}
\usepackage{enumerate}
\usepackage{theorem}
\newtheorem{theorem}{Theorem}[section]
\newtheorem{proposition}[theorem]{Proposition}
\newtheorem{lemma}[theorem]{Lemma}
\newtheorem{corollary}[theorem]{Corollary}

\theorembodyfont{\rmfamily}
\newtheorem{proof}{\textmd{\textit{Proof.}}}

\newtheorem{remark}[theorem]{Remark}
\newtheorem{example}[theorem]{Example}
\newtheorem{definition}[theorem]{Definition}

\makeatletter

\@addtoreset{equation}{section}
\makeatother

\newcommand{\qedd}{\hfill \Box}

\newcommand{\R}{\ensuremath{\mathbb{R}}}

\def\innerprod{\,\,\hbox to 4pt{\hrulefill}\kern-3.25pt\hbox{
     \vrule height7pt\,\,}}

\date{}
\begin{document}


\title{A variational problem for curves on Riemann-Finsler surfaces}
\author{S. V. Sabau, K. Shibuya}

\maketitle

\begin{abstract}
We study the variational problem for $N$-parallel curves on a Finsler surface by means of Exterior Differential Systems using Griffiths' method. We obtain the conditions when these curves are extremals of a length functional and write the explicit form of Euler-Lagrange equations for this type of variational problem.
\end{abstract}

\bigskip


\bigskip

2010 {\it Mathematics Subjects Classification. } Primary 49J40, Secondary 58B20.

\section{Introduction}

\quad The theory of variations is a central topic in optimization theory, mechanics, differential geometry and other fields of mathematics. In particular, the search for the extremals of arc length functionals for curves is the archetypal problem in Riemannian and Finsler geometry. 

It is well known that the extremals of the energy functional or a Riemannian or Finsler manifold are {\it geodesics} and that in fact these coincide with the extremals of the arc length functional for unit speed curves. An equivalent characterization is that a unit speed curve $\gamma$  on a Riemannian or Finsler manifold is a geodesic if and only if the tangent vector is parallel along $\gamma$ with respect to a certain connection. 

In the Riemannian case, this is equivalent to the fact that the normal vector along the unit speed curve $\gamma$ is also parallel with respect to the Levi-Civita connection, but this property do not extend to Finsler manifolds due to the dependency on direction of the Finslerian inner product.

We recall that, in a previous paper \cite{ISS2010}, we have encountered a family of curves $\gamma$ on a Finsler surface characterized by the property that the Finslerian normal vector $N$ is parallel along the curve $\gamma$ with respect to the Chern connection with reference vector $N$. We have called these curves {\it $N$-parallels} and have shown that these curves are deeply related to the Gauss-Bonnet type theorems in Finsler geometry proving in this way the importance of them. However, due to the difficulties of the problem setting and the complexity of computations, the 
study such curves on Finsler manifolds is not an easy task.

Motivated by these, we are interested in the following problem.

{\it Are the $N$-parallel curves extremals of some length functional?}

We give a first answer to this question  by studying the functional 
\begin{equation}\label{basic functional}
\mathcal L_N(\gamma):=\int_a^b \sqrt{g_N(T,T)}dt,
\end{equation}
where $T(t):=\dot\gamma(t)$ and $N(t)$ are the tangent and normal vectors along  the curve $\gamma:[a,b]\to M$, respectively. We determine the extremals of this functional for variations $\gamma:(-\varepsilon, \varepsilon)\times [a,b]\to M$, $(u,t)\mapsto \gamma(u,t)$, $\gamma(0,t)=\gamma(t)$ subject to some end points condition.

One might be tempted to consider this variational problem on the surface $M$ as in the classical variational problem for the usual arc length variation (see \cite{BCS2000}, \cite{Sh2001}). However, one can easily see that along the variation curve $\gamma(u,t)$, the variation vector field and the normal vector field both belong to the normal bundle
$$\{w\in T_\gamma M:T(t) \ 
{\rm  and} \  w(t) \ \textrm{ are linearly independent}\}$$
 and hence the treatment from the classical variational problem do not apply. Instead, we consider this variational problem "upstairs" on the indicatrix bundle $\Sigma$ by Griffiths' formalism and obtain in this way the Euler-Lagrange equations for this variational problem.

The novelty of the present research lies in the following:

\begin{enumerate}
\item We have formulated the variational problem for the functional \eqref{basic functional} and we have computed the corresponding Euler-Lagrange equations.
\item We have shown that the extremals of the 
functional \eqref{basic functional} 
form a family of curves on the base manifold $M$ that are different from the usual geodesics and have determined the conditions they coincide to the 
$N$-parallels.
 This one of the main differences between the arc length variational problem in Riemannian and Finsler settings. 
\item We have successfully used the Griffiths' formalism  for variational problems based on exterior differential systems in order to solve an intractable variational problem in the classical setting.
\end{enumerate}
 
More precisely, Griffiths' formalism in the calculus of variations is a very powerful method for a geometrical study, using exterior differential systems theory, of variational problems (classical and with constrains, first and higher order) developed in \cite{G1983}.

We consider that this is one of the best existing theories concerning the theory of variational problems, however unfortunately not yet popular enough amongst researchers. One of the first application of this method is the study of the functional $\frac{1}{2}\int {k^2}ds$ where Griffiths' formalism clarifies all the aspects of the problem (see \cite{BG1986}).

We have applied Griffiths' approach to calculus of variations to a variational problem in Finsler geometry and obtained important results concerning $N$-parallels. This is another evidence that this approach is a very powerful tool that deserves more attention. 

In the present paper we have successfully applied Griffiths' formalism in the study of a problem of Finsler geometry, but we point out that this method can be applied in the study of other variational problems in Riemannian geometry or in the general theory of calculus of variations. In future we intend to use this approach to the study of other variational problems. 

\bigskip

Here is the structure of our paper. We review the basics of Finsler surfaces in Section 2 and present the Griffiths' formalism in Section 3. In Section 4 we describe the geometry of $N$-parallel curves on a Finsler surface and express these curves as integral manifolds of an exterior differential system as well as second order differential equations. 

Section 5 contains the study of variational problem for the functional 
\eqref{basic functional}. Here we obtain the Euler-Lagrange equations for this functional (Theorem \ref{thm: N EL eq}) and we prove that these equations are necessary and sufficient conditions for the extremals of our functional (Theorem \ref{thm: nec and suffic cond}). The relation of the extremal curves with the geodesic curvature is given and this leads us to a new class of Finsler surfaces for which the $N$-parallels and extremals of the functional 
\eqref{basic functional} coincide. 

Finally we discuss the $N$-extremals for some special Finsler surfaces in Section 6. 

\bigskip

{\bf Acknowledgements.} We express our gratitude for many interesting discussions to J. Itoh and H. Shimada. We thank to N. Boonnan for reading an initial version of this paper.

\section{Riemann-Finsler surfaces}

A Finsler surface
 is a pair $(M,F)$ where $F:TM\to [0, \infty)$ is $\mathbf 
C^\infty$ on $\widetilde{TM}:=TM\backslash \{0\}$ and whose restriction to 
each tangent plane $T_xM$ 
is a Minkowski norm (see \cite{BCS2000}, \cite{Sh2001} or \cite{SS2007} for details).

A Finsler structure $(M,F)$ on a surface $M$ is equivalent to a smooth
hypersurface $\Sigma\subset TM$ for which the canonical projection
$\pi:\Sigma\to M$ is a surjective submersion and having the property
that for each $x\in M$, the $\pi$-fiber $\Sigma_x=\pi^{-1}(x)$ is
a strictly convex smooth curve including the origin $O_x\in T_xM$. 

In order to study the geometry of the surface $(M,F)$ we consider the pull-back bundle $\pi^*TM$ with base manifold $\Sigma$ and fibers 
$(T_xM)\vert_u$, where $u\in\Sigma$ such that $\pi(u) = x$ (see \cite{BCS2000}, Chapter 2 for details). 
The vector bundle $\pi^*TM$ has a distinguished global section $l:=\frac{y^i}{F(y)}\frac{\partial}{\partial x^i}$.

Using this section, one can construct a positively oriented $g$-orthonormal 
frame $\{e_1,e_2\}$ for $\pi^*TM$, by putting $e_2:=l$, where $g=g_{ij}(x,y)dx^i\otimes dx^j$ is 
the induced Riemannian metric on the fibers of $\pi^*TM$. The frame $\{u;e_1,e_2\}$ for any $u\in\Sigma$ is a globally defined $g$-orthonormal frame field 
for $\pi^*TM$ called the {\bf Berwald frame}. 

Locally, we have
\begin{equation*}
\begin{split}
e_1&:= \cfrac{1}{\sqrt g}\biggl(\cfrac{\partial F}{\partial 
y^2}\cfrac{\partial}{\partial x^1}-\cfrac{\partial F}{\partial 
y^1}\cfrac{\partial}{\partial x^2}\biggr)=m^1\cfrac{\partial}{\partial 
x^1}+m^2\cfrac{\partial}{\partial x^2}\quad ,\\
e_2&:= \cfrac{y^1}{F}\cfrac{\partial}{\partial 
x^1}+\cfrac{y^2}{F}\cfrac{\partial}{\partial 
x^2}=l^1\cfrac{\partial}{\partial 
x^1}+l^2\cfrac{\partial}{\partial x^2},
\end{split}
\end{equation*}
where $g$ is the determinant of the Hessian matrix $g_{ij}$.

The corresponding dual coframe is locally given by
\begin{equation*}
\begin{split}
\omega^1&=\cfrac{\sqrt g}{F}(y^2dx^1-y^1dx^2)=m_1dx^1+m_2dx^2\\
\omega^2&= \cfrac{\partial F}{\partial y^1}dx^1+\cfrac{\partial F}{\partial 
y^2}dx^2=l_1dx^1+l_2dx^2.
\end{split}
\end{equation*}

Next, one defines a moving coframing $(u;\omega^1,\omega^2,\omega^3)$ on $\pi^*TM$, orthonormal with respect to the
Riemannian metric on $\Sigma$ induced by the Finslerian metric $F$, where $u\in \Sigma$ and 
$\{\omega^1,\omega^2,\omega^3\}\in T^*\Sigma$. The
moving equations on this frame lead to the so-called Chern
connection. This is an almost metric compatible, torsion free connection of the
vector bundle $(\pi^*TM,\pi,\Sigma)$. 

Indeed, by a theorem of Cartan it follows that the coframe $(\omega^1,\omega^2,\omega^3)$ must satisfy the following structure equations
\begin{equation}\tag{2.2}
\begin{split}
d\omega^1&=-I\omega^1\wedge\omega^3+\omega^2\wedge\omega^3\\
d\omega^2&=-\omega^1\wedge\omega^3\\
d\omega^3&=K\omega^1\wedge\omega^2-J\omega^1\wedge\omega^3.
\end{split}
\end{equation}

The functions $I,J,K$ are smooth functions on $\Sigma$ called the {\bf invariants} of the Finsler structure $(M,F)$ in the sense of Cartan's equivalence problem (see for example \cite{BCS2000}, \cite{Br1997}).

On a Finsler manifold $(M,F)$ the metric
		\begin{equation}
		g=g_{ij}(x,y)dx^i\otimes dx^j
		\end{equation}
		is a family of inner products in each tangent space $T_xM$ viewed as a 
vector space, parameterized by rays $ty$, ($t>0$) which emanate from origin. 
This is actually a Riemannian metric on $\pi^*TM$. Moreover,
		\begin{equation}
		\hat{g}=g_{ij}(x,y)dy^i\otimes dy^j
		\end{equation}
		is a non-isotropic Riemannian metric on $T_xM$ viewed 
as differentiable manifold, which is invariant along each ray and possibly 
singular at origin.

This implies that on the vector bundle $\pi^*TM$ there exists a unique torsion-free and almost metric compatible connection $\nabla:C^\infty(T\Sigma)\otimes C^\infty(\pi^*TM)\to C^\infty(\pi^*TM)$, given by
\begin{equation}\tag{2.3}
 \nabla_{\hat{X}}Z=\{\hat{X}(z^i)+z^j\omega_j^{\ i}(\hat{X})\}e_i,
\end{equation}
where $\hat{X}$ is a vector field on $\Sigma$, $Z=z^ie_i$ is a section of $\pi^*TM$, and $\{e_i\}$ is the $g$-orthonormal frame field on $\pi^*TM$. 

The 1-forms $\omega_j^{\ i}$ define the {\bf Chern connection} of the Finsler structure $(M,F)$, where 
\begin{equation}\tag{2.4}\label{2.4}
(\omega_j^{\ i})=\left(\begin{array}{ccccc}{\omega_1{}^1} & {\omega_1{}^2}\\{\omega_2{}^1} 
&{\omega_2{}^2}\end{array}\right)=\left(\begin{array}{ccccc}{-I\omega^3} & 
{-\omega^3}\\{\omega^3} &{0}\end{array}\right),
\end{equation}
and $I:=A_{111}=A(e_1,e_1,e_1)$ is the {\bf Cartan scalar} for Finsler 
surfaces. Remark that $I=0$ is equivalent to the fact that the Finsler structure is 
Riemannian.

{\bf Remarks.}
\begin{enumerate}

\item We remark that the Chern connection gives a decomposition of the tangent bundle $T\Sigma$ by
\begin{equation*}
 T\Sigma=H\Sigma\oplus V\Sigma,
\end{equation*}
where the $H\Sigma$ is the horizontal distribution generated by $e_1,e_2$ and $V\Sigma$ is the vertical distribution generated by $\hat{e}_3$, where
$\hat{e}_1,\hat{e}_2,\hat{e}_3$ is the dual frame of the coframe $\omega^1,\omega^2,\omega^3$. 

\item For comparison, recall the structure equations of a Riemannian 
surface. They are obtained from (2.2) by putting $I=J=0$.

\item The scalar $K$ is called the {\bf Gauss curvature} of Finsler 
surface. In the case when $F$ is Riemannian, $K$ coincides with the usual 
 Gauss curvature of a Riemannian surface. 
\end{enumerate}

Differentiating again (2.2) one obtains the {\bf Bianchi identities}
\begin{equation}\tag{2.5}
\begin{split}
J&=I_2=\cfrac{1}{F}\biggl(y^1\cfrac{\delta I}{\delta x^1}+y^2\cfrac{\delta 
I}{\delta x^2}\biggr)\\
K_3&+KI+J_2=0,
\end{split}
\end{equation}
where 
$\{\frac{\delta}{\delta x^i},F\frac{\partial}{\partial y^i}\}$ is the adapted basis of $T\Sigma$, given by
\begin{equation*}
\frac{\delta}{\delta x^i}:=\frac{\partial}{\partial x^i}-N^j_i\frac{\partial}{\partial y^j}.
\end{equation*}
The functions $N_i^j$ are called the coefficients of the {\bf nonlinear connection} of $(M,F)$ (see \cite{BCS2000}, p. 33 for details).

The linear indices in $I_2$, $K_3$, $J_2$, etc. indicate differential terms with respect to 
$\omega_1,\omega_2,\omega_3$. For example 
$dK=K_1\omega^1+K_2\omega^2+K_3\omega^3$. The scalars $K_1$, $K_2$, $K_3$ 
are called the directional 
derivatives of $K$.

Nevertheless, remark that the scalars $I=I(x,y)$, $J=J(x,y)$, $K=K(x,y)$ and their derivatives 
live on $\Sigma$, not on $M$.

More generally, given any function $f:\Sigma \to \mathcal R$, one can write its 
differential in the form 
$df=f_1\omega_1+f_2\omega_2+f_3\omega_3$.

Taking one more exterior differentiation of this formula, one obtains the 
following {\bf Ricci identities}:
\begin{equation}\tag{2.6}
\begin{split}
f_{21}-f_{12} &= -Kf_3 \\
f_{32}-f_{23} &= -f_1 \\
f_{31}-f_{13} &= If_1+f_2+Jf_3.
\end{split}
\end{equation}

One defines {\bf the curvature} of the Finsler structure $(M,F)$  as usual by
\begin{equation}\tag{2.7}
\Omega_j^{\ i}=d\omega_j^{\ i}-\omega_j^{\ k}\wedge \omega_k^{\ i},
\end{equation}
where $i,j,k\in\{1,2,3\}$, and $\omega_j^{\ i}$ is the Chern connection matrix (2.4). It easily follows that the only essential entry in the matrix 
$\Omega_j^{\ i}$ is 
\begin{equation}\tag{2.8}\label{2.8}
\Omega_2^{\ 1}=d\omega_2^{\ 1}=d\omega^{3}=K\omega^1\wedge\omega^2-J\omega^1\wedge\omega^3.
\end{equation}


Recall that a Finsler surface is called {\bf Landsberg} if the invariant $J$ vanishes. Bianchi identities imply that in this case $I_2=0$ and $K_3=-KI$. A Finsler structure having $I_1=0$, $I_2=0$ is called a {\bf Berwald} surface (see \cite{BCS2000}, Lemma 10.3.1, p. 267 for details). 

It is known that Berwald surfaces are in fact Riemannian surfaces if $K\neq 0$ or locally Minkowski flat if $K=0$ (see \cite{Sz1981} and
 \cite{BCS2000}, p. 278).



\section{The variational problem in Griffiths' formulation}\label{sec: variational problem in Griffiths' formulation}

\subsection{The classical case}

We recall that if $L:\R\times TM\to \R$, $L=L(t;x,y)$ is a smooth non-autonomous Lagrangian function on the manifold $M$, then the length functional  associated to $L$ is given by 
\begin{equation}\label{Lagrangian functional}
\mathcal L:\Omega_{[a,b]}\to \R,\quad \gamma\mapsto \mathcal L(\gamma):=\int_a^b L(t;\gamma(t),\dot \gamma(t))dt,
\end{equation}
where $\Omega_{[a,b]}$ is the set of $C^k$ maps $\gamma:[a,b]\to M$.

The fundamental problem in the classical calculus of variations is to find the extremals of this functional with variations subject to some end point conditions. It is known (see for example \cite{AM2008}, \cite{MHSS2001} and many other textbooks)
that if the Lagrangian $L$ is non-degenerated, i.e. the Hessian matrix $\Bigl(\frac{\partial^2 L}{\partial y^i\partial y^j}\Bigr)$ is regular, then the extremals of $\mathcal L$ are the smooth maps  $\gamma:[a,b]\to M$ that satisfy the classical Euler-Lagrange equations
\begin{equation}
E_i(L):=\frac{\partial L}{\partial x^i}-\frac{d}{dt}\frac{\partial L}{\partial y^i}=0,\quad y^i=\frac{dx^i}{dt}.
\end{equation}

\begin{remark}
In the classical variational problem one considers fixed end points variations. A special case is the case of compactly supported variations of $\gamma$, i.e. variations $\gamma:(-\varepsilon,\varepsilon)\times (a,b)\to M$, $(u,t)\mapsto \gamma(u,t)$, $\gamma(0,t)=\gamma(t)$ such that outside a compact set $K\subset (a,b)$ the variation curve $\gamma(u,t)$ coincides with the base curve $\gamma$, i.e.
\begin{equation}
\gamma(u,t)=\gamma(0,t)=\gamma(t),\quad \forall t\in (a,b)\setminus K.
\end{equation}

The class of compact supported variations is a proper set of fixed end points variations, and therefore all the extremals of $\mathcal L$ with fixed end points variations are necessarily extremals of  $\mathcal L$ with compactly supported variations, i.e. they satisfy the same Euler-Lagrange equations, but the converse is not always true.

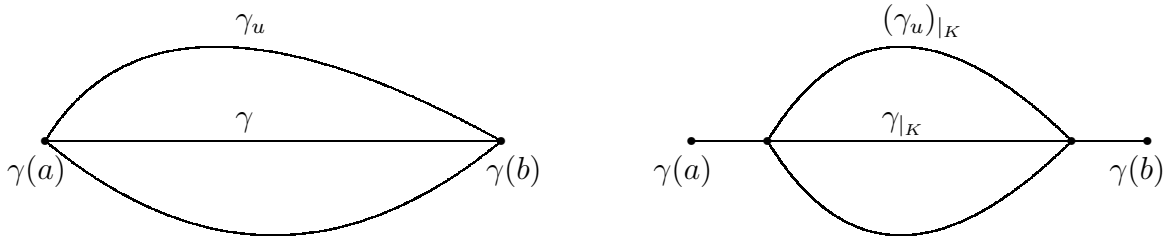
\setlength{\unitlength}{1cm} 
\begin{center}
\begin{picture}(9, 3)
\put(-3,1.5){\circle*{0.1}}
\put(3,1.5){\circle*{0.1}}
\put(-3,1.5){\line(1,0){6}}
\qbezier(-3,1.5)(-1.5,4)(3,1.5)
\qbezier(-3,1.5)(0,-1)(3,1.5)
\put(-3.5,1){$\gamma(a)$}
\put(2.8,1){$\gamma(b)$}
\put(-0.5,1.7){$\gamma$}
\put(-0.5,3){$\gamma_u$}
\put(5.5,1.5){\circle*{0.1}}
\put(11.5,1.5){\circle*{0.1}}
\put(5.5,1.5){\line(1,0){6}}
\put(6.5,1.5){\circle*{0.1}}
\put(10.5,1.5){\circle*{0.1}}
\qbezier(6.5,1.5)(8,4)(10.5,1.5)
\qbezier(6.5,1.5)(8,-1)(10.5,1.5)
\put(5,1){$\gamma(a)$}
\put(11,1){$\gamma(b)$}
\put(8,1.7){$\gamma_{|_K}$}
\put(8,3){$(\gamma_u)_{|_K}$}
\end{picture}
\end{center}

{\bf Figure 1.} Fixed end points variation (left) versus compactly supported variation (right).

\end{remark}

The classical variational problem can be reformulated in the language of exterior differential systems as follows (see \cite{G1983} and \cite{BG1986} for details). By putting $X:=\R\times TM=J^1(\R,M)$, that is the first order jets space on $M$, we have the canonical contact system on the manifold $X$:
\begin{equation}
\mathcal I:
\begin{cases}
\theta^1=dx^1-y^1dt\\
\theta^2=dx^2-y^2dt\\
\vdots \\
\theta^n=dx^n-y^ndt\\
\end{cases},
\end{equation}
where $(t;x^1,\dots,x^n,y^1,\dots,y^n)$ are the canonical local coordinates on the $(2n+1)$ dimensional manifold $X$. Obviously, for any $u=(t;x^1,\dots,x^n,y^1,\dots,y^n)\in X$, we have
\[
T^*_u M=\langle dt;dx^1,\dots,dx^n,dy^1,\dots,dy^n \rangle=\langle dt;\theta^1,\dots,\theta^n,dy^1,\dots,dy^n\rangle.
\]

One can see that if $\gamma:(a,b)\to M$ is a smooth mapping then there is a canonical lift to $X$ given by 
\begin{equation}
\hat\gamma:(a,b)\to X,\quad t\mapsto \hat\gamma(t)=(t;\gamma(t),\dot\gamma(t))
\end{equation}
such that $\hat\gamma$ is an integral manifold of the Pfaffian system with independence condition $(\mathcal I,\omega)$ defined on $X$, where $\mathcal I$ is the canonical contact system and $\omega=dt$, that is $\gamma^*(\mathcal I)=0$ and $\gamma^*(dt)\neq 0$. 

Moreover, we define the Lagrangian one form
\begin{equation}
\varphi:=L(t;x,y)dt,
\end{equation}
that is a well defined one form on $X$. 

Then the study of the functional \eqref{Lagrangian functional} can be reformulated as the study of the functional
\begin{equation}\label{EDS Lagrangian functional}
\mathcal L(\hat\gamma)=\int_{\hat\gamma}\varphi,
\end{equation}
where  $\hat\gamma$ is an integral manifold of the Pfaffian system with independence condition $(\mathcal I,\omega)$ defined on $X$.

The main problem of the calculus of variations reads as follows.

{\it Find the extremals of the functional \eqref{EDS Lagrangian functional} when $\hat\gamma$ varies through the set of integral manifolds of $(\mathcal I,\omega)$ of compact support.}

 These extremals project to the extremals of \eqref{Lagrangian functional} through the canonical projection $J^1(\R,M)\to M$. 

\begin{example}
Consider a Riemannian or Finslerian manifold $(M,F)$, that is $F:TM\to \R$ is an autonomous Lagrangian on the manifold $X:=J^1(\R,M)$ which is positive 1-homogeneous with respect to the variable $y$. In the Riemannian case $F(x,y)=\sqrt{g(y,y)}$, where $g$ is the Riemannian metric, and in the Finslerian case $F$ is the fundamental function. Taking the energy function of a Riemannian or Finslerian structure is also possible. 

The study of the functional 
\[
\mathcal L(\gamma)=\int_\gamma F dt
\]
is the main subject of calculus of variations on Riemannian or Finslerian manifolds (see any textbook on Riemannian or Finslerian manifolds). 

The problem can be studied "upstairs" on $TM$ or $J^1(\R,M)$ or "downstairs" directly on $M$. For the "upstairs" setting, the discussion above applies. In the "downstairs" setting $\mathcal I=(0)$. 

\end{example}

\subsection{The general case}

More generally, a {\it variational problem}  is the study of the functional
\begin{equation}\label{general functional}
\Phi:\mathcal V(\mathcal I,\omega)\to \R,\quad \Phi(\gamma_{|(a,b)})=\int_\gamma \varphi=\int_a^b \gamma^*\varphi,
\end{equation}
where $(\mathcal I,\omega)$ is a Pfaffian differential system of rank $s$ on the manifold $X$, $\gamma$ is a typical integral manifold of  
$(\mathcal I,\omega)$, i.e.
\begin{equation}
\gamma_{|(a,b)}\in \mathcal V(\mathcal I,\omega):=\{\gamma:(a,b)\to X\ :\ \gamma^*(\mathcal I)=0,\ \gamma^*(\omega)\neq 0\},
\end{equation}
and $\varphi$ is a one form on $X$ (here the curves that differ only by parameterization will be identified).

One can see that this is a natural generalization of the classical variational problem formulated in the language of exterior differential systems. The main problem of the calculus of variations is the same. {\it Describe the extremals of the functional $\Phi$, that is, determine the Euler-Lagrange equations of $\Phi$}. 

More precisely, if we denote by $T_\gamma \mathcal V(\mathcal I,\omega)$ the "tangent space" of $\mathcal V(\mathcal I,\omega)$ at $\gamma$, then we can consider the differential of \eqref{general functional}, i.e.
\begin{equation}
\delta \Phi_\gamma:T_\gamma \mathcal V(\mathcal I,\omega)\to \R,\quad \delta\Phi_\gamma(v)=\frac{d}{du}\Bigl( \int_{\gamma_u}\varphi  \Bigr)|_{u=0},
\end{equation}
where $\gamma_u\in \mathcal V(\mathcal I,\omega)$ is any compactly supported variation of $\gamma$ with $\gamma_0=\gamma$ and 
$v\in T_\gamma \mathcal V(\mathcal I,\omega)$ is the associated infinitesimal variation vector field defined along $\gamma$ corresponding to the variation $u\mapsto \gamma_u$. 

With these notations, the Euler-Lagrange equations of $\Phi$ are
\begin{equation}
\delta\Phi_\gamma(v)=0, \quad \forall v\in T_\gamma\mathcal V(\mathcal I,\omega).
\end{equation}

Integral curves $\gamma$ satisfying these equations are called the extremals of $\Phi$. 

We make a small digression here to point out that  the "tangent space" $ T_\gamma\mathcal V(\mathcal I,\omega)$, that is the space of smooth variation vector fields of $\gamma$, can be described to first order by the variational equations of an integral curve of the Pfaffian system $(\mathcal I,\omega)$ (see \cite{G1983}, \cite{H1992} for details). The variational equations of the integral curves of $(\mathcal I,\omega)$ are given by
\begin{equation}
\mathcal D_\gamma(v)=0, \qquad v\in T_{\gamma(t)} X\setminus \langle \gamma'(t)\rangle,
\end{equation}
where
\begin{equation}
\mathcal D_\gamma(v):=e_\alpha\otimes (v \innerprod d\theta^\alpha+d(v\innerprod \theta^\alpha))|_{\gamma}.
\end{equation}
Here $\{\theta^1,\theta^2,\dots,\theta^s\}$ is a local basis for  $\mathcal I$, and 
$\{e_1,e_2,\dots,e_s\}$ its dual frame field along $\gamma$. Locally, the tangent space $ T_\gamma\mathcal V(\mathcal I,\omega)$ is described by the equations
\begin{equation}
(v \innerprod d\theta^\alpha+d(v\innerprod \theta^\alpha)|_{\gamma}=0,\quad \alpha=1,2,\dots,s.
\end{equation}

We return to the Euler-Lagrange equations of the variational problem  $(\mathcal I, \omega;\varphi)$ on the manifold $X$. 
In the case $\mathcal I=(0)$, the characterization of extremals is fairy simple

\begin{proposition}{\rm (\cite{AM2008}, \cite{Br1987})}\label{prop:I=0 EL eq}
If $\gamma:(a,b)\to X$ is an extremal of the functional $\gamma\mapsto \int_\gamma \varphi$ with compactly supported variations, then for all $t\in (a,b)$, we have
\begin{equation}\label{I=0 EL eq}
\gamma'(t)\innerprod d\varphi |_{\gamma(t)}=0.
\end{equation}

Conversely, if $\gamma$ satisfies this condition, then $\gamma:(a,b)\to X$ is an extremal.

\end{proposition}

Remark that the left hand side of the formula above is a one form on $X$ defined along $\gamma$. 

The form \eqref{I=0 EL eq} of the Euler-Lagrange equations allows to associate another Pfaffian system $(J,\omega)$ to the variational problem  $(\mathcal I=(0), \omega;\varphi)$. Indeed, one can see that actually the Euler-Lagrange equations  \eqref{I=0 EL eq} give the integral manifolds of the Cartan system of the two form $\Psi:=d\varphi$ on $X$, namely
\begin{equation}
\mathcal C(\Psi):=\{v\innerprod d\varphi\ :\ v\textrm{ is any smooth compactly supported vector field on }X\}.
\end{equation}

It follows that the extremals of $(\mathcal I=(0), \omega;\varphi)$ are the characteristic curves of the two form $\Phi$ on $X$, that is $\gamma:(a,b)\to X$ satisfies the conditions 
\begin{enumerate}
\item $\gamma'(t) \innerprod  \Psi_{\gamma(t)}=0,\qquad \forall t\in (a,b)$;
\item $\gamma^*(\omega)\neq 0$.
\end{enumerate}

The Pfaffian system $(J,\omega):=(\mathcal C(\Psi),\omega)$ is called {\bf the Euler-Lagrange system} of the variational problem $(\mathcal I=(0), \omega;\varphi)$ on $X$ (see \cite{G1983}, \cite{BG1986}).

The characterization of the Euler-Lagrange equations by the Euler-Lagrange system can be extended to variational problems $(\mathcal I, \omega;\varphi)$ on $X$, when $\textrm{ rank } \mathcal I>0$ (see \cite{G1983}, \cite{BG1986}).

For the variational problem $(\mathcal I, \omega;\varphi)$, $\textrm{ rank }\mathcal I=s>0$, on a manifold $X$, one constructs the affine sub bundle
\begin{equation}
Z:=\mathcal I+\varphi\subset T^*X,
\end{equation}
that is $Z_x=\mathcal I_x+\varphi_x$ is an affine subspace of $T^*_xX$, for any $x\in X$. 

Locally $Z_U\simeq U\times \R^s$, where $U\subset X$ is an open set, in other words, we identify the pair $(x,\lambda)\in U\times \R^s$ with the one form
\begin{equation}
\psi_x:=\varphi_x+\sum_{\alpha=1}^s \lambda_i\theta^i_x\in T_x^*X,
\end{equation}
 where $\{\theta^1,\theta^2,\dots,\theta^s\}$ is a local basis for  $\mathcal I$ over $U$. By this identification it results that  $\psi$ is the canonical one form on $Z$ obtained by the restriction of the 
 canonical one form on $T^*X$ to $Z$. The reason d'etre of this construction lies in the following fundamental result.

\begin{theorem}{\rm (\cite{Br1987})}\\
Let $(\mathcal I, \omega;\varphi)$ be a variational problem on $X$ and let 
$(\widetilde{I}=(0), \omega;\psi)$, be the associated variational problem on $Z$ constructed above. Then the canonical projection $\pi:Z\subset TX\to X$ maps extremals of $(\widetilde{I}=(0), \omega;\psi)$ to the extremals of $(\mathcal I, \omega;\varphi)$.
\end{theorem}

In other words, the Euler-Lagrange equations for the variational problem $(\mathcal I, \omega;\varphi)$ on $X$ can be defined as the Euler-Lagrange equations of the associated variational problem $(\widetilde{I}=(0), \omega;\psi)$ on $Z$ obtained as in Proposition \ref{prop:I=0 EL eq}, where $\omega$ on $Z$ is obtained by the pullback of $\omega$ on $X$ by the canonical projection.

For this reason the Pfaffian system $(\mathcal J:=\mathcal C(\Psi),\omega)$, typically restricted to a sub manifold 
  $Y\subset Z$, is called {\bf the Euler-Lagrange system} associated to the variational problem $(\mathcal I, \omega;\varphi)$ on $X$. Any variational problem for curves can be formulated in this setting. 

\begin{remark}\label{Problems}
However, one must pay attention to the following problems when working in this formalism.
\begin{enumerate}
\item In general the differential system $(\mathcal C(\Psi),\omega)$ is not a Pfaffian system with independence condition on $Z$, so we need to construct the involutive prolongation of $(\mathcal C(\Psi),\omega)$ on a  a sub manifold 
  $Y\subset Z$ (see \cite{G1983}, \cite{BG1986} for details).
\item The Euler-Lagrange equations for the variational problem $(\mathcal I, \omega;\varphi)$, that we describe above, are sufficiently conditions for $\gamma$ to be extremals. However, in the case $\textrm{ rank } \mathcal I>0$ these are not always necessary conditions. Indeed, it is known that any regular extremal curve of $\Phi$ is a solution of the Euler-Lagrange equations (\cite{H1992}), where $\gamma$ regular means that it is a generic integral curve of a bracket generating differential system. Recall that a distribution $D=\mathcal I^\perp\subset TX$ is bracket generating if and only if for any $v\in T_xM$ there exist the vector fields $X_1,X_2,\dots,X_k\in D_x\subset T_xX$ such that
\begin{equation}
[X_1,[X_2,\dots,[X_{k-1},X_k],\dots ]_x=v,
\end{equation}
where $x$ is a generic point in $X$.
\end{enumerate}
\end{remark}

\section{The normal lift of a curve}

\quad In this section we recall the normal lift of a curve from \cite{ISS2010}.

Let us consider a smooth (or piecewise $C^\infty$) curve $\gamma:[0,r]\to M$ with the tangent vector $\dot{\gamma}(t)=T(t)$, parameterized such that $F(\gamma(t),\dot{\gamma}(t))=1$.

We construct the  normal vector field $N$ along $\gamma$, i.e. 
\begin{equation}\label{orthog_conditions}
\begin{split}
& g_N(N,N)=1\\
& g_N(N,T)=0\\
& g_N(T,T)=\sigma^2(t),
\end{split}
\end{equation}
where $\sigma(t)$ is a scalar function non constant along $\gamma$. 

This lead us to the {\bf normal lift} $\hat\gamma^\perp$ of $\gamma$ to $\Sigma$ defined by
\begin{equation}
\begin{split}
 \hat\gamma^\perp:[0,&r]\to\Sigma\\
& t\mapsto \hat\gamma^\perp(t)=(\gamma(t),N(t)).
\end{split}
\end{equation}

The tangent vector to the normal lift is therefore
\begin{equation}
\begin{split}
\hat{T}^\perp(t) & = \hat{\dot{\gamma}}^\perp(t)=\frac{d}{dt}\hat{\gamma}^\perp(t)=\dot{\gamma}^i(t)\frac{\partial}{\partial x^i}_{\vert_{(x,N)}}+\frac{d}{dt}N^i(t)\frac{\partial}{\partial y^i}_{\vert_{(x,N)}} \\
& =T^i(t)\frac{\delta}{\delta x^i}_{\vert_{(x,N)}}+(D_T^{(N)}N)^i\frac{\partial}{\partial y^i}_{\vert_{(x,N)}},
\end{split}
\end{equation}
where
\begin{equation}
D_T^{(N)}U=(D_T^{(N)} U)^i\cdot\frac{\partial}{\partial x^i}_{\vert_{\gamma(t)}}=\Bigl[\frac{dU^i}{dt}+T^jU^k\Gamma^i_{jk}(x,N) \Bigr]\cdot\frac{\partial}{\partial x^i}_{\vert _{\gamma(t)}},
\end{equation}
for any $U=U^i(x)\frac{\partial}{\partial x^i}$ vector field along $\gamma$, where $\Gamma^i_{jk}$ are the Chern connection coefficients, i.e. $\omega_i^j=\Gamma_{ik}^jdx^k$, that is the covariant derivative along $\gamma$ with reference vector $N$.

It follows 
\begin{equation}
\begin{split}
& g_N(D_T^{(N)}N,N)=0,\\
& g_N(D_T^{(N)}T,N)+g_N(T, D_T^{(N)}N)=0,\\
& g_N(D_T^{(N)}T,T)=
\sigma(t)\frac{d\sigma}{dt}-A^{(N)}(T,T,D_T^{(N)}N).
\end{split}
\end{equation}

Since $\{N,T\}$ is a basis of $T_\gamma M$, we also obtain
\begin{equation}
D_T^{(N)}T=k_{T}^{(N)}(t)N+B(t)T,
\end{equation}
where
 \begin{equation}
B(t)=\frac{1}{\sigma(t)}\frac{d\sigma(t)}{dt}-\frac{1}{\sigma^2(t)}A^{(N)}(T,T,D_T^{(N)}N).
\end{equation}

We can define the notion of $N$-parallel of a Finsler structure.
\begin{definition}
A curve $\gamma$ on the surface $M$, in Finslerian natural parameterization,
 is called an
{\bf N-parallel} of the Finslerian structure $(M,F)$ if and only if the normal vector field is parallel along $\gamma$, namely we have
\begin{equation}
D_T^{(N)}N=0.
\end{equation}
\end{definition}

It follows that the normal lift of an  $N$-parallel curve $\gamma$
on $M$ is given by 
$\hat{T}^\perp=T^i\frac{\delta}{\delta x^i}_{\vert_{(x,N)}}$. 

\begin{remark}
\begin{enumerate}
\item If $\gamma$ is an  $N$-parallel, then we have
\begin{equation}
\begin{split}
& g_N(D_T^{(N)}T,N)=0\\
& g_N(D_T^{(N)}T,T)=
\sigma(t)\frac{d\sigma}{dt}.
\end{split}
\end{equation}
\item The curve $\gamma$ is an  $N$-parallel if and only if 
$\nabla_{\hat{T}^\perp}l=D_T^{(N)}N=0$. This implies that
      $g_N(\nabla_{\hat{T}^\perp}l,l)=0$, i.e. $\nabla_{\hat{T}^\perp}l$
      is orthogonal to the indicatrix.
\end{enumerate}
\end{remark}

In case of an arbitrary curve $\gamma$ on $M$, from 
\begin{equation*}
g_N(D_T^{(N)}N,N)=0, \qquad g_N(T,N)=0,
\end{equation*}
it follows that the vector $D_T^{(N)}N$ is proportional to $T$,
i.e. there exists a non-vanishing function $k_{T}^{(N)}(t)$ such that
\begin{equation}
D_T^{(N)}N=-\frac{k_{T}^{(N)}(t)}{\sigma^2(t)}T,\qquad \sigma(t)\neq 0.
\end{equation}

The function $k_{T}^{(N)}(t)$ will be called the {\bf $N$-parallel curvature
of $\gamma$}. The minus sign is put only in order to obtain the
same formulas as in the classical theory of Riemannian manifolds.

In other words, we have
\begin{equation}
g_N(D_T^{(N)}N,T)=-g_N(D_T^{(N)}T,N)=-k_{T}^{(N)}(t).
\end{equation}

Since $\{N,T\}$ is a basis, we also obtain
\begin{equation}
D_T^{(N)}T=k_{T}^{(N)}(t)N+B(t)T,
\end{equation}
where we put 
 \begin{equation}
B(t)=\frac{1}{\sigma(t)}\frac{d\sigma(t)}{dt}-\frac{1}{\sigma^2(t)}A_{\vert_{(x,N)}}(T,T,D_T^{(N)}N).
\end{equation}

\begin{corollary}
A curve $\gamma$ on $M$ is $N$-parallel if and only if its $N$-parallel curvature $k_{T}^{(N)}$ vanishes.
\end{corollary}


By making use of the cotangent map of $\hat{\gamma}^{\perp}$ we compute
\begin{equation}
\begin{split}
& \hat{\gamma}^{\perp *}\omega^1\frac{\partial}{\partial
 t}=\omega^1(\hat{T}^\perp)_{\vert_{(x,N)}}
=\frac{\sqrt{g}}{F}(N^2T^1-T^2N^1)=\sigma(t)\\
& \hat{\gamma}^{\perp *}\omega^2\frac{\partial}{\partial
 t}=\omega^2(\hat{T}^\perp)_{\vert_{(x,N)}}=g_N(N,T)=0\\
& \hat{\gamma}^{\perp *}\omega^3\frac{\partial}{\partial
 t}=\omega^3(\hat{T}^\perp)_{\vert_{(x,N)}}
=\frac{\sqrt{g}}{F}
\Bigl[ N^2(D_T^{(N)}N)^1-N^1(D_T^{(N)}N)^2
\Bigr]=-\frac{k_{T}^{(N)}(t)}{\sigma(t)}.
\end{split}
\end{equation}

Therefore we obtain
\begin{equation}\label{3.14}
\begin{split}
& \hat{\gamma}^{\perp *}\omega ^1=\sigma(t)dt \\
& \hat{\gamma}^{\perp *}\omega^2=0 \\
& \hat{\gamma}^{\perp *}\omega^3=-\frac{k^{(N)}_{T}}{\sigma(t)}dt.
\end{split}
\end{equation}

If we denote by $\{\hat e_1, \hat e_2, \hat e_3\}$ the dual frame 
of the orthonormal coframe $\{\omega^1, \omega^2, \omega^3\}$, then we obtain that the tangent vector to the normal
lift of $\hat{\gamma}^{\perp}$ is 
\begin{equation}\tag{3.15}\label{3.15}
\hat{T}^\perp=\sigma(t)\hat{e}_1-\frac{k^{(N)}_{T}}{\sigma(t)}\ \hat{e}_3\in<\hat{e}_1,\hat{e}_3>,
\end{equation}
where $<\hat{e}_1,\hat{e}_3>$ is the 2-plane generated by $\hat{e}_1$, $\hat{e}_3$.
\bigskip


Remark that in the case when $\gamma$ is an $N$-parallel, we have
\begin{equation}\label{3.16}
D_T^{(N)}T=\frac{1}{\sigma(t)}\frac{d\sigma(t)}{dt}T,
\end{equation}
and
\begin{equation}
\begin{split}
& \hat{\gamma}^{\perp *} \omega ^1=\sigma(t)dt \\
& \hat{\gamma}^{\perp *}\omega^2=0 \\
& \hat{\gamma}^{\perp *}\omega^3=0.
\end{split}
\end{equation}

 Finally, we remark that the tangent vector to the normal
lift of an $N$-parallel is 
\begin{equation}
\hat{T}^\perp=\sigma(t)\hat{e}_1.
\end{equation}

\begin{remark}

We can rewrite $\eqref{3.16}$ as 
\begin{equation}\label{7.3}
\frac{d^2\gamma^i}{dt^2} + \Gamma^i_{jk}\big(\gamma(t), N(t)\big)\frac{d\gamma^j}{dt}\frac{d\gamma^k}{dt} = 
\frac{d}{dt}\Bigl[ \log \sigma(t) \Bigr] \frac{d\gamma^i}{dt},
\end{equation}
where $N(t) = N\big(\gamma(t), \dot\gamma(t)\big)$ from \eqref{orthog_conditions}.

An initial condition can be given by
\begin{equation}\label{7.4}
\begin{split}
\gamma^i(t_0) = 0,\\
\dot\gamma^i(t_0) = T_0^i,
\end{split}
\end{equation}
with $i = 1, 2$ and corresponding the normal initial condition
\begin{equation}
\begin{split}
\gamma^i(t_0) = 0,\\
N^i(t_0) = N_0,
\end{split}
\end{equation}
where $N_0$ are given as solutions of $\eqref{orthog_conditions}$ for $T = T_0$.

Then, by a similar argument as in the case of geodesics, we know from the general theory of ODEs that $\eqref{7.3}$ with initial conditions $\eqref{7.4}$ have unique solutions.

\end{remark}



\section{Variational problem for the $N$-lift}


We will formulate a variational problem in Griffiths' formalism for our setting by specifying the manifold $X$ and the Pfaffian with independence condition $(\mathcal I,\omega)$ by means of the 3-manifold $\Sigma$ with the coframe $\{\omega^1,\omega^2,\omega^3\}$ generated by the Finsler surface $(M,F)$.

First, remark that equations \eqref{3.14} imply
\begin{proposition}
Let $\gamma:[a,b]\to M$ be a smooth curve on $M$. 
Then its normal lift $\hat{\gamma}^\perp$ to $\Sigma$ 
is an integral manifold of the Pfaffian system with independence condition $(\mathcal I,\omega^1)$ on the manifold $\Sigma$, where $\mathcal I=\{\omega^2\}$.
\end{proposition}

Clearly the projection to $M$ of any integral curve of $(\mathcal I,\omega^1)$ is a curve on $M$.

We will consider the variational problem in Griffiths' formalism 
 $(\mathcal I,\omega^1;\varphi)$ on the manifold $\Sigma$, with
\begin{equation}
\varphi=\omega^1.
\end{equation}

More precisely, we consider the functional 
\begin{equation}\label{N-var functional}
\Phi:\mathcal V(\mathcal I,\omega^1)\to \R,\quad \Phi(\hat\gamma^\perp)=\int_{\hat\gamma^\perp} \omega^1,
\end{equation}
where $\hat\gamma^\perp:(a,b)\to \Sigma$ is a typical integral manifold of the rank one Pfaffian system $(\mathcal I,\omega^1)$.

The extremals of this functional are called the $N$-{\it extremals} of the Finsler surface $(M,F)$.











In order to compute the Euler-Lagrange equations of this variational problem, we follow Griffiths' recipe in Section 3, and consider the manifold $Z:=\Sigma\times \R$, where $\R$ has the coordinate $\lambda$ and put
\begin{equation}
\psi:=\varphi+\lambda \omega^2
\end{equation}
on $Z$.

The exterior derivative $\Psi=d\psi$ is given by
\begin{equation}
\Psi=d\omega^1+d\lambda\wedge \omega^2+\lambda d\omega^2=
-(I+\lambda)\omega^1\wedge\omega^3+\omega^2\wedge\omega^3+d\lambda \wedge \omega^2.
\end{equation}

A coframe on $Z$ is given by 
\begin{equation}
\{\omega^1,\omega^2,\omega^3;d\lambda\},
\end{equation}
and the corresponding frame
\begin{equation}
\{\hat{e}_1,\hat{e}_2,\hat{e}_3,\frac{\partial}{\partial\lambda}\},
\end{equation}
where we use $\{\hat e_1,\hat e_2,\hat e_3\}$ for the dual frame of $\{\omega^1,\omega^2,\omega^3\}$ on $\Sigma$.

It follows that the Cartan system $\mathcal C(\Psi)$ is given by
\begin{equation}
\begin{split}
\mathcal C(\Psi)& :=\{\frac{\partial}{\partial\lambda}\innerprod \Psi,
\hat{e}_2\innerprod \Psi,\hat{e}_3\innerprod \Psi\}
=\{\omega^2,-d\lambda+\omega^3,(\lambda+ I)\omega^1-\omega^2  \}\\
& =\{\omega^2,d\lambda-\omega^3,(\lambda+ I)\omega^1  \}
\end{split}
\end{equation}

However, remark that ($\mathcal C(\Psi),\omega^1)$ is  not a Pfaffian system with independence condition on $\Sigma$ because 
$(\lambda+ I)\omega^1\in \mathcal C(\Psi)$.

Hence we need an involutive prolongation of  ($\mathcal C(\Psi),\omega^1)$ on a submanifold on $Z$. Following first part of Remark \ref{Problems}
we consider the submanifold $Z_1:=\{\lambda=-I\}\subset Z$, and therefore 
\begin{equation}\label{Cartan sys on Z_1}
\mathcal C(\Psi)_{|Z_1}=\{\omega^2, dI+\omega^3\}.
\end{equation}

We can see that the iterative construction stops for $k=1$, and therefore, if we put $Y:=\{\lambda=-I\}$, 
then $(\mathcal C(\Psi)_{|Y},\omega_{|Y})$ is  a Pfaffian system with independence condition on $Z_1=Z_2=\dots=Y$.

Remark that by identifying the graph of a function with its domain of definition, we can see that actually $Y=\Sigma$.

We obtain

\begin{theorem}
\begin{enumerate}
\item The Euler-Lagrange differential system of the variational problem $(I,\varphi;\omega)$ is 
\begin{equation}
(dI+\omega^3)|_{\hat{\gamma}^\perp}=0, 
\end{equation}
\item and the corresponding Euler-Lagrange equation is 
\begin{equation}\label{EL_{eq}}
(I_3\circ\hat\gamma^\perp+1)k_{T}^{(N)}=(I_1\circ \hat\gamma^\perp)\sigma^2.
\end{equation}
\end{enumerate}
\end{theorem}
\begin{proof}
From \eqref{Cartan sys on Z_1} we obtain 1. Moreover, using now the equations \eqref{3.14}, the Euler-Lagrange equation follows. 

$\qedd$
\end{proof}

Moreover we have
\begin{theorem}\label{thm: nec and suffic cond}
The Euler-Lagrange equation \eqref{EL_{eq}} is necessary and sufficient condition for the extremals of the functional \eqref{N-var functional}.
\end{theorem}
\begin{proof}

The sufficiency is obvious from construction. In order to prove the necessity, since any curve on a contact manifold is regular,  it is enough to show that 
the distribution $D=\mathcal I^\perp$ is bracket generating. Since $\mathcal I=\{\omega^2\}$ on $\Sigma$ it follows that $D=\langle \hat e_1, \hat e_3\rangle$. Cartan formula
$d\omega(X,Y)=X(\omega(Y))-\omega([X,Y])-Y(\omega(X))$
implies
\begin{equation*}
 [\hat e_1,\hat e_2]=-K\hat e_3,\ 
 [\hat e_2,\hat e_3]=-\hat e_1,\ 
 [\hat e_3,\hat e_1]=-I\hat e_1-\hat e_2-J\hat e_3.
\end{equation*}

It follows immediately that $D_1:=[D,D]=\langle \hat e_1,  \hat e_2, \hat e_3\rangle=T\Sigma$ and therefore $D$ is bracket generating. 

$\qedd$

\end{proof}

We remark that 
\begin{equation}
\psi_{|Y}=\omega^1-I\omega^2,\quad \Psi_{|Y}=
-I_1\omega^1\wedge\omega^2+(I_3+1)\omega^2\wedge\omega^3.
\end{equation}

Since $\dim Y=2\times 1+1=3$ it follows $m=1$, and hence we obtain
\begin{equation}
\psi_{|Y}\wedge \Psi_{|Y}^m=(I_3+1)\omega^1\wedge\omega^2\wedge\omega^3\neq 0,
\end{equation}
for $I_3\neq 1$. That is, in this case, the variational problem $(\mathcal I, \omega^1;\varphi)$ is {\it non-degenerate}.

The end points conditions are given by
\begin{equation}\label{end_point}
K=\{\omega^1,\omega^2\}
\end{equation}
and the variational problem is well-posed with reduced momentum space 
$$Q=\Sigma\slash_{\{\omega^1=0,\omega^2=0\}}=M.$$

In our case an admissible variation is a map 

\begin{equation}
\hat{\gamma}^\perp:[a,b]\times [0,\varepsilon]\to \Sigma, (t,u)\mapsto 
\hat{\gamma}^\perp(t,u)
\end{equation}
such that each $u$-curve in the variation, namely 
$\hat{\gamma}^\perp_u:[a,b]\to \Sigma$ is integral manifold of $(\mathcal I,\omega)$.

Then the admissible variations satisfying the end point conditions \eqref{end_point} means
\begin{equation}
\hat\gamma^{\perp*}(\omega^1)=\hat\gamma^{\perp*}(\omega^2)=0 \ {\rm on } \ \{a,b\}\times [0,\varepsilon].
\end{equation}


 That is this corresponds to varying a curve $\gamma$ in $Q=M$ keeping its end points fixed in the usual sense.

In other words, we obtain
\begin{theorem}\label{thm: N EL eq}
\begin{enumerate}
\item 
 If $I_1\circ \hat\gamma^\perp= 0$, then Euler-Lagrange equation implies  $k_{T}^{(N)}=0$, provided  $I_3\circ\hat\gamma^\perp+1\neq 0$, i.e. in this case the $N$-extremals are the $N$-parallels. Conversely, if the $N$-parallels are $N$-extremals, then $I_1\circ \hat\gamma^\perp= 0$.
\item  If $I_1\circ \hat\gamma^\perp \neq 0$, then Euler-Lagrange equation implies 
\begin{equation}
k_{T}^{(N)}=\frac{I_1\circ \hat\gamma^\perp}{I_3\circ\hat\gamma^\perp+1},
\end{equation}
provided  $I_3\circ\hat\gamma^\perp+1\neq 0$, i.e. in this case the $N$-extremals are those curves on $\Sigma$ whose geodesic curvature is given above. 
\end{enumerate}
\end{theorem}



\section{Special Finsler surfaces}
\subsection{Berwald surfaces}
We discuss the {\it Berwald surfaces} case.  It is known that there are only two cases:
\begin{enumerate}
\item $K\neq 0$, that is $(M,F)$ is Riemannian surface,
\item $K=0$, that is $(M,F)$ is locally Minkowski plane.
\end{enumerate}

We discuss first the {\it Riemannian} case $(M,a)$. Indeed, in this case, $I=0$ and the Euler-Lagrange equations 
\eqref{EL_{eq}} read $k_{T}^{(N)}=0$. In other words, the solutions of the Euler-Lagrange equations 
\eqref{EL_{eq}} are the $N$-parallel curves, or equivalently, the solutions of the following SODE:

\begin{equation}
\frac{d^2\gamma^i}{dt^2} + \gamma^i_{jk}\big(\gamma(t)\big)\frac{d\gamma^j}{dt}\frac{d\gamma^k}{dt} = 
\frac{d}{dt}\Bigl[ \log \alpha(t) \Bigr] \frac{d\gamma^i}{dt},
\end{equation}
where $\alpha(t)=\sqrt{a(T,T)}$, that is the usual equation of Riemannian geodesics in arbitrary parameterization.

In the locally Minkowski case we have 
\begin{equation}
I_{1}=0,\quad I_{2}=0, \quad K=0.
\end{equation}

\begin{lemma}\label{lem: I_3 constant}
Let $(M,F)$ be a Finsler surface. If $I_3=$constant everywhere on $\Sigma$, then $I_3$ must vanish on $\Sigma$.
\end{lemma}
\begin{proof}

If $I_3=$constant everywhere on $\Sigma$ it follows that the scalar $I$ must be constant along every indicatrix curve $\Sigma_x\subset T_xM$. On the other hand, it is known that the average value of the Cartan scalar, over the indicatrix $\Sigma_x$ must be zero (see for example \cite{BCS2000}, p. 85), that is
\begin{equation}\label{I average value}
\int_0^LI(t)dt=0,
\end{equation}
where $I(t)$ is the Cartan scalar evaluated over the indicatrix and $L$ is the Riemannian length of $\Sigma_x$. 

If $I(t)=c$=constant, then \eqref{I average value} implies $cL=0$, that is $c=0$ since the indicatrix length $L$ cannot be zero.

$\qedd$
\end{proof}

In this case, we have
\begin{proposition}
If $(M,F)$ is a Minkowski surface, then the solutions of the Euler-Lagrange equations 
\eqref{EL_{eq}} coincide to the $N$-parallel curves and they differ from the usual geodesics. 
\end{proposition}

\begin{proof}
From Theorem \ref{thm: N EL eq} and Lemma \ref{lem: I_3 constant} it follows that $N$-extremals and $N$-parallels must coincide for a Minkowski surface. 

One can see  that Finslerian geodesics and the $N$-extremals cannot coincide on  $\Sigma$ because their tangent vectors $\hat e_2$ and  $\hat e_1$, respectively, are linearly independent. 
$\qedd$
\end{proof}


\subsection{A new class of Finsler spaces}

We define a new class of special Finsler spaces as follows.

\begin{definition}
A Finsler surface $(M,F)$ that satisfies the following conditions
\begin{equation}\label{S-surface}
I_{1}=0,\qquad I_{3}\neq 0
\end{equation}
is called an {\it $S$-Finsler surface}.
\end{definition}

\begin{remark}
Any Berwald manifold is an $S$-manifold.
\end{remark}

Let us remark that on an $S$-manifold we have the Bianchi equations
\begin{equation}
\begin{split}
& dI=I_2\omega^2+I_3\omega^3\\
& dK=K_1\omega^1+K_2\omega^2-(KI+I_{22})\omega^3\\
& dI_2=-KI_3\omega^1+I_{22}\omega^2+I_{23}\omega^3.
\end{split}
\end{equation}

The tableau of the free derivatives has Cartan characters $s_1=3$, $s_2=2$, $s_3=1$, the Cartan test yields $7=s_1+2s_2+3s_3$ and therefore  the system is involutive. The Cartan-K\"ahler theory implies that, modulo diffeomorphism, such structures depend on 2 functions of 2 variables.  

Returning to our variational problem, we have
\begin{proposition}\label{prop: geom1}
If $(M,F)$ is an $S$-Finsler surface, then the integral lines of codimension one foliation 
\begin{equation}
\{\omega^{2}=0,\omega^{3}=0\}
\end{equation}
coincide with the $N$-extremals.
\end{proposition}

\begin{remark}
The geometrical meaning of the $S$-Finsler surfaces it is now clear. These are those Finsler surfaces where the $N$-parallels coincide with the $N$-extremals. 
\end{remark}

Finding more examples and applications of the present theory are  subject of forthcoming research.



\bigskip

\begin{center}
Sorin V. SABAU\\

\bigskip

Department of Mathematics, Tokai University\\
Sapporo, 
005\,--\,8601 Japan

\medskip
{\tt sorin@tspirit.tokai-u.jp}
\end{center}

\begin{center}
Kazuhiro SHIBUYA\\

\bigskip

Graduate School of Science,
Hiroshima University \\
Hiroshima, 739\ --\ 8521 
Japan

\medskip
{\tt shibuya@hiroshima-u.ac.jp}
\end{center}

\end{document}